\documentclass[11pt]{article}
\usepackage{amscd, amsmath, amssymb, amsfonts, amsthm}
\usepackage{euscript}
\usepackage{graphicx}
\usepackage[square, comma, sort&compress, numbers]{natbib}

\let\oldref=\ref
\def\ref#1{(\oldref{#1})}

\def\theenumi{\roman{enumi}}

\theoremstyle{plain}
\newtheorem{prop}{Proposition}[section]
\newtheorem{thm}[prop]{Theorem}
\newtheorem{lem}[prop]{Lemma}
\newtheorem{cor}[prop]{Corollary}
\newtheorem{ques}[prop]{Question}

\theoremstyle{definition}
\newtheorem{ex}[prop]{Example}
\newtheorem{defn}[prop]{Definition}
\newtheorem{rem}[prop]{Remark}

\newif\ifSuppressEndOfProof\SuppressEndOfProoffalse

\makeatletter
\def\p@figure{Fig. }
\def\p@enumii{}

\makeatother
\newcommand{\cc}{\EuScript{C}}
\newcommand{\ca}{\mathfrak{A}}
\newcommand{\cd}{\EuScript{D}}
\newcommand{\MS}{\mathcal{MS}}

\newcommand{\sms}{\mathcal{S}\MS}
\newcommand{\F}{\mathrm{Facets}}
\newcommand{\link}{\mathrm{link}}

\newcommand{\V}{\mathrm{V}}
\newcommand{\rnd}{\partial}
\newcommand{\N}{\mathrm{N}}

\newcommand{\z}{\mathbb{Z}}

\newcommand{\ifof}{if and only if }
\newcommand{\se}{\subseteq}
\newcommand{\su}{\subsetneq}
\newcommand{\dis}{\oplus}

\newcommand{\sm}{\setminus }
\renewcommand{\iff}{\Leftrightarrow}
\newcommand{\give}{$\Rightarrow$}
\newcommand{\rgive}{$\Leftarrow$}
\renewcommand{\l}{\left}
\renewcommand{\r}{\right}
\newcommand{\rg}{\rangle}
\renewcommand{\lg}{\langle}
\newcommand{\p}[1]{\left(#1\right)}
\newcommand{\f}[2]{\frac{#1}{#2}}
\newcommand{\tohi}{\emptyset}
\renewcommand{\b}[1]{\overline{#1}}
\newcommand{\tl}[1]{\widetilde{#1}}

\newcommand{\blt}{\bullet}
\renewcommand{\Im}{\mathrm{Im}\,}
\newcommand{\Tor}{\mathrm{Tor}}
\notag
\begin{document}
%%%%%%%%%%%%%%%%%%%%%%%%%%%%%%%%%%%%%%%%%%%%%%%%%%%%%%%%%%%%%%%%%%%%%%%%%%%%%%%%%%%%%%%%%%%%%%%%%%%%%%%%%%%%%%%%%%

\title{Chordality of Clutters with Vertex Decomposable Dual and Ascent of Clutters}
\author{Ashkan Nikseresht\footnote{While this paper was under review, the author's affiliation has been
 changed to: Department of Mathematics, Shiraz University, 71457-13565, Shiraz, Iran} \\
\it\small Department of Mathematics, Institute for Advanced Studies in Basic Sciences (IASBS),\\
\it\small P.O. Box 45195-1159, Zanjan, Iran.\\
\it\small E-mail: ashkan\_nikseresht@yahoo.com\\
}
\date{}
\maketitle

\begin{abstract}
In this paper, we consider the generalization of chordal graphs to clutters proposed by Bigdeli, et al in  J. Combin.
Theory, Series A (2017). Assume that $\cc$ is a $d$-dimensional uniform clutter. It is known that if $\cc$ is
chordal, then $I(\b\cc)$ has a linear resolution over all fields. The converse has recently been rejected, but the
following question which poses a weaker version of the converse is still open: ``if $I(\b\cc)$ has linear quotients,
is $\cc$ necessarily chordal?''. Here, by introducing the concept of the ascent of a clutter, we split this question
into two simpler questions and present some clues in support of an affirmative answer. In particular, we show that if
$I(\b\cc)$ is the Stanley-Reisner ideal of a simplicial complex with a vertex decomposable Alexander dual, then $\cc$
is chordal.
\end{abstract}


%%%%%%%%%%%%%%%%%%%%%%%%%%%%%%%%%%%%%%%%%%%%%%%%%%%%%%%%%%%%%%%%%%%%%%%%%%%%%%%%%%%%%%%%%%%%%%%%%%%%%%%%%%%%%%%%%%
                                        \section{Introduction}
%%%%%%%%%%%%%%%%%%%%%%%%%%%%%%%%%%%%%%%%%%%%%%%%%%%%%%%%%%%%%%%%%%%%%%%%%%%%%%%%%%%%%%%%%%%%%%%%%%%%%%%%%%%%%%%%%%
In this paper all rings are commutative with identity, $K$ is a field and $S=K[x_1,\ldots, x_n]$ is the polynomial ring
in $n$ indeterminates over $K$. Given an arbitrary ideal $I$ of $S$, we can get a monomial ideal by taking its initial
ideal or its generic initial ideal with respect to some monomial order. Many of the properties of $I$ is similar or at
least related to its (generic) initial ideal, especially if $I$ is graded (see, for example, \cite[Section 3.3 \&
Corollary 4.3.18]{hibi}). Also if $I$ is a monomial ideal of $S$ then its polarization is a squarefree monomial ideal,
again with algebraic properties similar to $I$ (see \cite[Section 1.6]{hibi}). Therefore, if we know the algebraic
properties of squarefree monomial ideals well, then we can understand many of the algebraic properties of much larger
classes of ideals.

On the other hand, squarefree monomial ideals have a combinatorial nature and correspond to combinatorial
objects such as simplicial complexes, graphs, clutters and posets. Many researchers have tried to
characterize algebraic properties of squarefree monomial ideals, using their combinatorial counterparts, see
for example \cite{our chordal} and Part III of \cite{hibi} and their references. In this regard, two
important objectives are to classify combinatorially squarefree monomial ideals which are Cohen-Macaulay or
those which have a linear resolution (either over every field or over a fixed field). Indeed these two tasks
are equivalent under taking Alexander dual.

A well-known theorem of Fr\"oberg states that if $I$ is a squarefree monomial ideal generated in degree two, then $I$
has a linear resolution \ifof $I$ is the edge ideal of the complement of a choral graph. Motivated by this, many have
tried to generalize the concept of chordality to clutters or simplicial complexes in such a way that at least one side
of Fr\"oberg's theorem stays true in degrees $>2$ (see, for instance, \cite{w-chordal, CF1,chordal}). One of the most
promising such generalizations seems to be the concept of chordal clutters presented in \cite{chordal}. In
\cite{chordal} it is shown that many other classes of ``chordal clutters'' defined by other researchers, including the
class defined in \cite{w-chordal} which we call W-chordal clutters, is strictly contained in the class defined in
\cite{chordal}. They also show that if a clutter is chordal, then the circuit ideal of its complement has a linear
resolution over every field. Several clues were presented in \cite{chordal,simp ord} supporting the correctness of the
converse (see \cite[Question 1]{chordal}). But recently a counterexample to the converse was presented by Eric Babson
(see Example \ref{counter Ex}). Despite this it is still unknown whether the following statement which is a weaker
version of the converse is true or not: ``if the circuit ideal of the complement of a clutter has linear quotients,
then that clutter is chordal.'' Moreover, in \cite{simp ord}, it is shown that many numerical invariants of the ideal
corresponding to a chordal clutter can be combinatorially read off the clutter.

Here, after presenting a brief review of the main concepts and setting the notations, in Section 3 we prove that if the
Alexander dual of the clique complex of a clutter $\cc$ is vertex decomposable, then $\cc$ is chordal. As clique
complexes of W-chordal clutters are vertex decomposable, this generalizes the results of Subsection 3.1 of
\cite{chordal}. Then in Section 4, we define the notion of the ascent of a clutter and show how we can use this concept
to divide the question ``is $\cc$ chordal, given that the circuit ideal of the complement of $\cc$ has  linear
quotients?'' into two simpler questions.
%%%%%%%%%%%%%%%%%%%%%%%%%%%%%%%%%%%%%%%%%%%%%%%%%%%%%%%%%%%%%%%%%%%%%%%%%%%%%%%%%%%%%%%%%%%%%%%%%%%%%%%%%%%%%%%%%%
                                        \section{Preliminaries, notations and a counterexample}
%%%%%%%%%%%%%%%%%%%%%%%%%%%%%%%%%%%%%%%%%%%%%%%%%%%%%%%%%%%%%%%%%%%%%%%%%%%%%%%%%%%%%%%%%%%%%%%%%%%%%%%%%%%%%%%%%%

\paragraph{Algebraic background.}
Suppose that $I$ is a graded ideal of $S$ considered with the standard grading. This grading induces a natural grading
on $\Tor_i^S(K,I)$. We denote the degree $j$ part of $\Tor_i^S(K,I)$ by $\Tor_i^S(K,I)_j$ and its dimension over $K$ is
denoted by $\beta_{ij}^S(I)= \beta_{ij}(I)$. These $\beta_{ij}$'s are called the graded \emph{Betti numbers} of $I$. If
there is a $d\geq 0$ such that $\beta_{ij}(I)=0$ for $j \neq i+d$, it is said that $I$ has a \emph{linear} (or
\emph{$d$-linear}) \emph{resolution}. Also we say that $I$ has \emph{linear quotients} with respect to an ordered
system of homogenous generators $f_1, \ldots, f_t$, if for each $i$ the ideal $\lg f_1, \ldots, f_{i-1} \rg: f_i$ is
generated by linear forms. It is known that if $I$ is generated in degree $d$ and has linear quotients with respect to
some system of homogenous generators, then $I$ has a $d$-linear resolution (see \cite[Proposition 8.2.1]{hibi}).

Now assume that $I$ is a squarefree monomial ideal, that is, $I$ is generated by some squarefree monomials. Then $I$
has a unique smallest generating set consisting of squarefree monomials, say $x_{F_1}, \ldots, x_{F_t}$ where
$x_F=\prod_{i\in F} x_i$ for $F\se [n]=\{1,\ldots,n\}$. In this case, if $I$ has linear quotients with respect to a
permutation of this minimal system of generators, we simply say that $I$ has linear quotients. Also a specific order of
minimal generators of $I$, for which the aforementioned colon ideals are linear, is called an \emph{admissible order}.
According to \cite[Corollary 8.2.4]{hibi}, $x_{F_1}, \ldots, x_{F_t}$ is an admissible order for the ideal they
generate, \ifof for each $i$ and all $j<i$, there is a $l\in F_j\sm F_i$ and a $k<i$ such that $F_k\sm F_i=\{l\}$. For
more details on these algebraic concepts see \cite{hibi}.

\paragraph{Clutters.}
A \emph{clutter} $\cc$ on the vertex set $V= \V(\cc)$ is a family of subsets of $V$ which are pairwise incomparable
under inclusion. We call the elements of $\cc$ \emph{circuits}. For any subset $F$ of $V$ we set $\dim F=|F|-1$. If all
circuits of $\cc$ have the same dimension $d$, we say that $\cc$ is a \emph{$d$-dimensional uniform} clutter or a
\emph{$d$-clutter} for short. If $v\in \V(\cc)$, then by $\cc-v$ we mean the clutter on $\V(\cc)\sm \{v\}$ with
circuits $\{F\in \cc|v\notin F\}$. For simplicity, we write for example $ab$, $Ex$ or $Eab$ instead of $\{a,b\}$,
$E\cup\{x\}$ or $E\cup\{a,b\}$ for $a,b,x\in V$ and $E\se V$. If we assume that $V=[n]$, then $I(\cc)$ is defined as
the ideal of $S$ generated by $\{x_F|F\in \cc\}$.

From now on, we always assume that $\cc$ is a $d$-clutter with $|\V(\cc)|=n$.  We call $\cc$ \emph{complete}, when all
$d$-dimensional subsets of $V=\V(\cc)$ are in $\cc$. A \emph{clique} of $\cc$ is a subset $A$ of $V$ such that
$\cc_{A}$ is complete, where $\cc_A=\{F\in \cc|F\se A\}$ is the induced $d$-clutter on the vertex set $A$. The
complement $\b \cc$ of $\cc$ is the $d$-clutter with the same vertex set as $\cc$ and circuits $\{F\se V| \dim F=d, F
\notin \cc\}$.

The set $\{e\se V|\dim e=d-1, \exists F\in \cc\quad e\se F\}$ is called the set of \emph{maximal subcircuits} of $\cc$
and is denoted by $\MS(\cc)$. For a $(d-1)$-dimensional subset $e$ of $V$, the \emph{closed neighborhood} $N_\cc[e]$ is
defined as $e\cup \{v\in V| ev\in \cc\}$. If $e\in \MS(\cc)$, then $\cc-e$ means the clutter on the vertex set $V$ with
circuits $\{F\in \cc|e\not\se F\}$.

A maximal subcircuit $e$ of $\cc$ is called a \emph{simplicial maximal subcircuit}, when $N_\cc[e]$ is a clique. We
denote the set of all simplicial maximal subcircuits of $\cc$ by $\sms(\cc)$. If for each $1\leq i< t$ there is an
$e_i\in \sms(\cc_{i-1})$, where $\cc_0=\cc$ and $\cc_i= \cc_{i-1}-e_i$, such that $\cc_t$ has no circuits, then $\cc$
is called \emph{chordal} (see \cite[Section 3]{chordal}). In the case that $d=1$ (that is, $\cc$ is graph) this notion
coincides with the usual notion of chordal graphs. Theorem 3.3 of \cite{chordal} states that if $\cc$ is chordal and
not complete, then $I(\b \cc)$ has a linear resolution over every field. There is an example showing that the converse
is not true, that is, there is a non-chordal clutter $\cc$ with $I(\b\cc)$ having a linear resolution over every field
(see Example \ref{counter Ex}). But it is still unknown whether there is a non-chordal clutter $\cc$ with $I(\b\cc)$
having linear quotients. For further reference, we label this statement:

\def\theenumi{\Alph{enumi}}
\begin{enumerate}
%\item \label{A} If $I(\b \cc)$ has a linear resolution over every field, then $\cc$ is chordal.
\item \label{B} If $I(\b \cc)$ has linear quotients, then $\cc$ is chordal.
\end{enumerate}
\def\theenumi{\roman{enumi}}

\paragraph{Simplicial complexes.}
A \emph{simplicial complex} $\Delta$ on the vertex set $V=\V(\Delta)$ is a family of subsets of $V$ (called \emph{faces
}of $\Delta$) such that if $A\se B\in \Delta$, then $A\in \Delta$. The dimension of $\Delta$ is defined as $\dim
\Delta= \max_{F\in \Delta}\dim F$. The set of maximal faces of $\Delta$ which are called \emph{facets} is denoted by
$\F(\Delta)$. If $|\F(\Delta)|=1$, then $\Delta$ is called a \emph{simplex}.

If all facets of $\Delta$ have the same dimension, we say that $\Delta$ is \emph{pure}. In this case
$\F(\Delta)$ is a $d$-dimensional uniform clutter. Also if $\cd$ is a clutter, then $\lg \cd \rg$ denotes the
simplicial complex $\Delta$ with $\F(\Delta)= \cd$. It should be clear that $\lg \F(\Delta) \rg= \Delta$ and
$\F(\lg \cd \rg) =\cd$, for any simplicial complex $\Delta$ and any (not necessarily uniform) clutter $\cd$.
Another simplicial complex associated to a $d$-clutter $\cc$ is the \emph{clique complex} $\Delta(\cc)$ of
$\cc$ defined as the family of all subsets $L$ of $\V(\cc)$ with the property that $L$ is a clique in $\cc$.
Note that all subsets of $\V(\cc)$ with size $\leq d$ are cliques by assumption.

For a face $F$ of $\Delta$, we define $\link_\Delta F= \{G\sm F|F\se G\in \Delta\}$, which is a simplicial complex on
the vertex set $V\sm F$. Also if $v\in V$, $\Delta-v$ is the simplicial complex on the vertex set $V\sm \{v\}$ with
faces $\{F\in \Delta|v\notin F\}$.

Assuming that $\V(\Delta)=[n]$, the ideal of $S$ generated by $\{x_F|F$ is a minimal non-face of $\Delta\}$ is called
the \emph{Stanley-Reisner ideal} of $\Delta$ and is denoted by $I_\Delta$. When $\f{S}{I_\Delta}$ is a Cohen-Macaulay
ring, $\Delta$ is said to be Cohen-Macaulay over $K$.
 %and $\Delta|_L =\{G\in \Delta| G\se L\}$. Moreover, we call $\Delta$ \emph{$d$-complete} when $\Delta$ has
%all $(d+1)$-subsets of $[n]$. We call a family $\cc$ of $r$-dimensional faces of $\Delta$ an \emph{$r$-cycle}, if $\cc$
%is a cycle as a clutter.

Let $A$ be a commutative ring with identity and $-1\leq d\leq \dim \Delta$ and denote by $\tl{C}_d(\Delta)=
\tl{C}_d(\Delta, A)$ the free $A$-module whose basis is the set of all $d$-dimensional faces of $\Delta$. Consider the
$A$-homomorphism $\rnd_d: \tl{C}_d(\Delta) \to \tl{C}_{d-1}(\Delta)$ defined by
 $$\rnd_d(\{v_0, \ldots, v_d\}) =\sum_{i=0}^d (-1)^i \{v_0, \ldots, v_{i-1},v_{i+1},\ldots, v_d\},$$
where $v_0<\cdots <v_d$ for a fixed total order $<$ on $\V(\Delta)$. Then $(\tl{C}_\blt, \rnd_\blt)$ is a complex of
free $A$-modules and $A$-homomorphisms called the \emph{augmented oriented chain complex} of $\Delta$ over $A$. We
denote the $i$-th homology of this complex by $\tl{H}_i(\Delta; A)$.

By the \emph{Alexander dual} of a simplicial complex $\Delta$ we mean $\Delta^\vee= \{\V(\Delta) \sm F| F\se \V(\Delta)
,\,F\notin \Delta\}$ and also we set $\cc^\vee=\{\V(\cc)\sm F\big| F\in \b{\cc}\}$. Then it follows from the
Eagon-Reiner theorem (\cite[Theorem 8.1.9]{hibi}) and the lemma below that $I(\b{\cc})$ has a linear resolution over
$K$, \ifof $\lg \cc^\vee \rg$ is Cohen-Macaulay over $K$. For more details on simplicial complexes and related
algebraic concepts the reader is referred to \cite{hibi}. We frequently use the following lemma in the sequel without
any further mention.

\begin{lem}[{\cite[Lemma 1.1]{our chordal}}]\label{transition}
Let $\cc$ be a $d$-clutter. Then
\begin{enumerate}
\item $I(\b\cc)=I_{\Delta(\cc)}$;
\item $\lg \cc^\vee \rg= (\Delta(\cc))^\vee$.
\end{enumerate}
\end{lem}

When working with both clutters and simplicial complexes, one should notice the differences between the
similar concepts and notations defined for these objects. For example, we say that a $d$-dimensional
simplicial complex $\Delta$ is \emph{$i$-complete} to mean that $\Delta$ has all possible faces of dimension
$i$, where $i\leq d$. But when we are talking about a $d$-clutter $\cc$, there is no need to mention the
dimension and say that $\cc$ is $i$-complete, since we must have $i=d$ which is implicit in $\cc$ and it is
completely meaningless to say that $\cc$ is $i$-complete for $i<d$. Therefore we just say that $\cc$ is
complete. Another difference arises from the concept of Alexander dual which is illustrated in the following
example.
\begin{ex}\label{dual-exam}
Let $\cc=\{125,235,345\}$ be a 2-clutter on $[5]$, $\Gamma=\lg \cc \rg$ and $\Delta=\Delta(\cc)$. Note that
$14 \notin \Gamma$ because it is not contained in any facet of $\Gamma$. But by definition, 14 is a clique of
$\cc$ and hence $14\in \Delta=\lg \{125,235,345,13,14,24\} \rg$. Now for example $245\in \bar\cc$ and hence
$[5]\sm 245=13\in \cc^\vee$. Indeed, $\cc^\vee=\{13,15,23,24,25,35,45\}$, which by \ref{transition} is equal
to the set of facets of $\Delta^\vee$. But as $14\notin \Gamma$, we see that $[5]\sm 14= 235\in \Gamma^\vee$.
In fact, $\lg \cc \rg^\vee= \Gamma^\vee= \lg \{235,245,135\}\rg \neq \lg \cc^\vee \rg$.

\end{ex}

Recall that $\Delta_W=\{F\in \Delta|F\se W\}$ where $W\se V$.
\begin{thm}[{Fr\"oberg \cite{froberg}}]\label{Fro}
The ideal $I_\Delta$ has a $t$-linear resolution over $K$, \ifof $\tl H_i(\Delta_W; K)=0$ for all $W\se \V(\Delta)$ and
$i\neq t-2$.
\end{thm}
Note that in the case that $\Delta= \Delta(\cc)$, then since all possible faces of dimension less than $d$
are in $\Delta$, it follows that $\tl H_i(\Delta_W; K)=0$ for all $W\se \V$ and $i<d-1$. Hence $I(\b \cc)$
has a $(d+1)$-linear resolution over $K$, \ifof $\tl H_i(\Delta_W; K)=0$ for all $W\se \V(\Delta)$ and $i\geq
d$. Using this we can state an example of a non-chordal clutter $\cc$ with $I(\b\cc)$ having a linear
resolution over every field. This example is due to Eric Babson who visited IASBS as a lecturer in the first
Research School on Commutative Algebra and Algebraic Geometry in 2017.

\begin{ex}[Babson]\label{counter Ex}
Let $\cc$ be the 2-clutter shown in \oldref{Fig-Dunce} which is a triangulation of the dunce hat. Then $\Delta=
\Delta(\cc)$ is obtained by adding the missing edges (1-dimensional subsets of vertices) to $\Delta'= \lg \cc \rg$,
since $\cc$ has no cliques on more that 3 vertices. Therefore, $\tl H_t(\Delta; K)= \tl H_t(\Delta'; K)$ for every
$t\geq 2$. It is known that the dunce hat is a contractible space and hence $0= \tl H_t(\Delta'; K)$ for all $t\geq 2$.
Also it is easy to verify that for any $W\su [8]$, $\cc_W$ is chordal (indeed, any proper subclutter of $\cc$ has an
edge which is contained in exactly one triangle and hence is simplicial) and thus has a linear resolution over every
field. Therefore, $\tl H_t(\Delta_W; K)= \tl H_t(\Delta(\cc_W); K)=0$ for all $t\geq 2$. Consequently, by F\"oberg's
theorem $I_\Delta= I(\b\cc)$ has a linear resolution over every field.
\end{ex}
\begin{figure}[!ht]
  \centering
  \includegraphics[scale=.4]{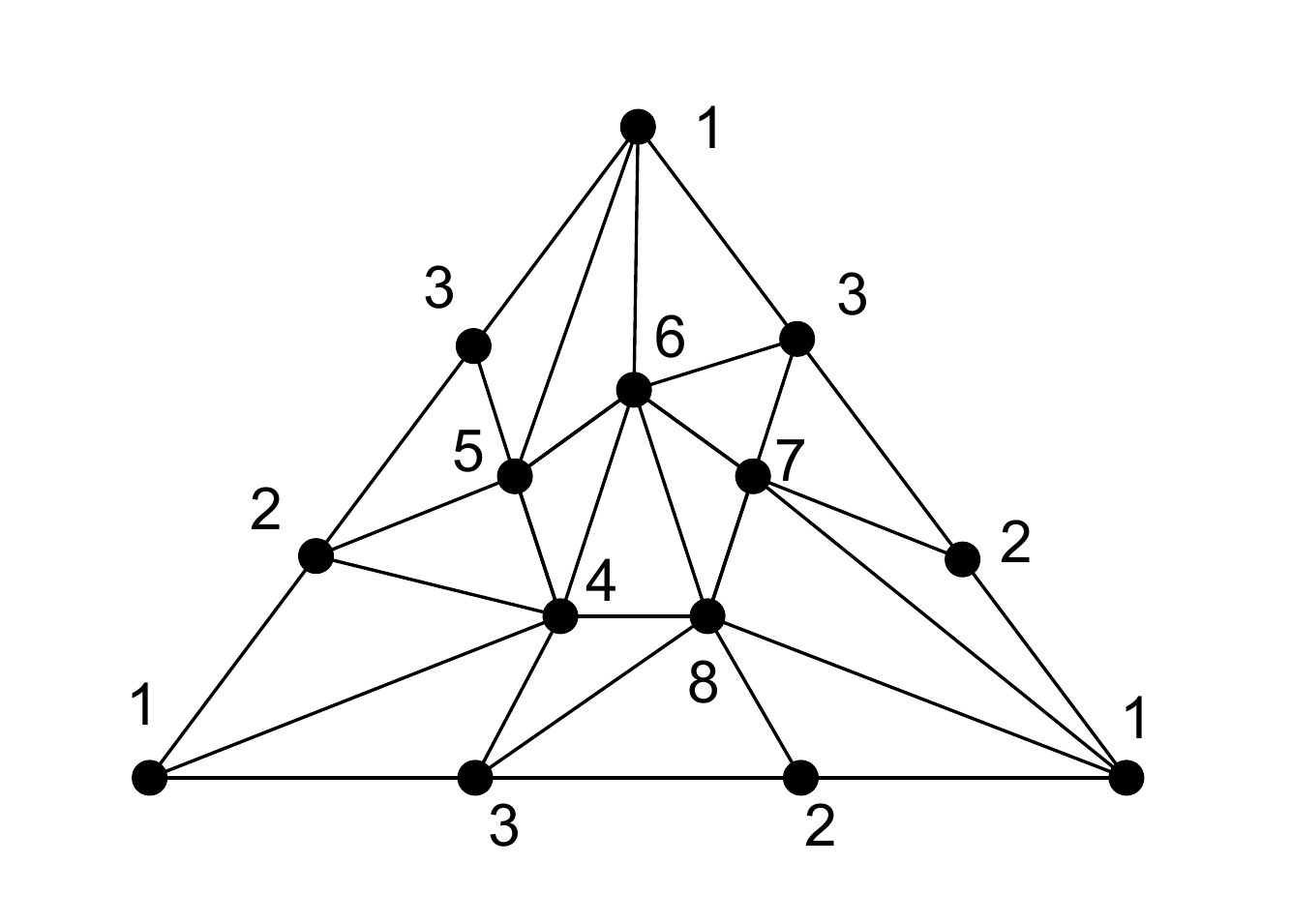}
  \caption{A triangulation of the dunce hat; the circuits are the triangles}\label{Fig-Dunce}
\end{figure}

%%%%%%%%%%%%%%%%%%%%%%%%%%%%%%%%%%%%%%%%%%%%%%%%%%%%%%%%%%%%%%%%%%%%%%%%%%%%%%%%%%%%%%%%%%%%%%%%%%%%%%%%%%%%%%%%%%
                         \section{Clutters with vertex decomposable dual are chordal}
%%%%%%%%%%%%%%%%%%%%%%%%%%%%%%%%%%%%%%%%%%%%%%%%%%%%%%%%%%%%%%%%%%%%%%%%%%%%%%%%%%%%%%%%%%%%%%%%%%%%%%%%%%%%%%%%%%

A vertex $v$ of a nonempty simplicial complex $\Delta$ is called a \emph{shedding vertex}, when no face (or
equivalently facet) of $\link_\Delta(v)$ is a facet of $\Delta-v$.  Recall that a nonempty simplicial complex
$\Delta$ is called \emph{vertex decomposable}, when either it is a simplex or there is a shedding vertex
$v\in V(\Delta)$ such that both $\link_\Delta v$ and $\Delta- v$ are vertex decomposable. This concept was
first introduced in \cite{provan} in connection with the Hirsch conjecture which has applications in the
analysis of the simplex method in linear programming.

\begin{lem}
Suppose that $\Delta$ is a pure $d$-dimensional simplicial complex. Then $v$ is a shedding vertex \ifof
$\Delta-v$ is pure with $\dim(\Delta-v)=d$.
\end{lem}
\begin{proof}
(\give): Assume that $\Delta-v$ is not pure or $\dim(\Delta-v)\neq d$. Since every facet of $\Delta-v$ has
dimension $d-1$ or $d$, there is a facet $F$ of $\Delta-v$ such that $\dim F= d-1$.  As $F\in \Delta$ and
$\Delta$ is pure, we see that $F\su F'$ for some $F'\in \F(\Delta)$. Hence $F'=Fv\in \Delta$ and $F\in
\link_\Delta(v)$, which contradicts the shedding property for $v$.

(\rgive): Every face of $\link_\Delta(v)$ has dimension $\leq d-1$ and by assumption $\Delta-v$ is pure and
of dimension $d$. Thus $\link_\Delta(v) \cap \F(\Delta-v)=\tohi$ and $v$ is a shedding vertex.
\end{proof}

\begin{ex}
Suppose that $\Gamma$ is as in Example \oldref{dual-exam}. By choosing 1 and then 4 as the shedding vertex,
we see that $\Gamma$ is vertex decomposable. Here $\Gamma-5=\lg 12, 23, 34 \rg$ is pure with dimension
$1=\dim (\Gamma)-1$ and $\Gamma-2=\lg 15, 345 \rg$ is not pure. So by the previous lemma, neither 5 is a
shedding vertex of $\Gamma$ nor 2. Also $\Gamma^\vee$ which is indeed isomorphic to $\Gamma$ is vertex
decomposable. Now assume that $\Delta=\lg 12,34 \rg$ which is a 1-dimensional simplicial complex. Then
$\Delta-1=\lg 2, 34\rg$ is not pure, so 1 is not a shedding vertex of $\Delta$. Similarly we see that
$\Delta$ has no shedding vertex and is not vertex decomposable.
\end{ex}

It is well-known that if $\Delta$ is pure and vertex decomposable, then it is shellable and Cohen-Macaulay (see
\cite{provan}), hence $I_{\Delta^\vee}$ has a linear resolution and in fact linear quotients. So if $\cc$ is the
clutter with $I(\b \cc)=I_{\Delta^\vee}$ and if statement \ref{B} is true, $\cc$ should be chordal. In this section, we
prove that this is indeed the case. By \ref{transition}, in the above situation we have $\Delta= (\Delta(\cc)) ^\vee=
\lg \cc ^\vee \rg$. First we need some lemmas. In the sequel, we assume that $\cc$ is a $d$-clutter and $|\V(\cc)|=n$.
Also recall that the \emph{pure $i$-skeleton} of a simplicial complex $\Delta$, is the simplicial complex whose facets
are $i$-dimensional faces of $\Delta$.

\begin{lem}\label{dual of F(link)}
Let $\cc$ be a $d$-clutter. Suppose that $\Delta=\lg \cc \rg$ and $\Gamma=\lg \cc^\vee \rg$. Also assume that $v\in
\V(\Gamma)$ ($=\V(\Delta)$) and $\Gamma-v$ is pure of dimension $=\dim(\Gamma)$ (that is, $v$ is a shedding vertex of
$\Gamma$) and set $\cd= \F(\link_\Delta(v))$. Then
\begin{enumerate}
\item \label{dual of F1} $\link_\Delta(v)$ is the pure $(d-1)$-skeleton of $(\Gamma-v)^\vee$;
\item \label{dual of F2} $\lg \cd^\vee \rg= \Gamma-v$;
\item \label{dual of F3} $e\in \sms(\cd)$ \ifof $ev\in \sms(\cc)$.
\end{enumerate}
\end{lem}
\begin{proof}
\ref{dual of F1} Suppose that $F\in (\Gamma-v)^\vee$ and $\dim F=d-1$. Then $A=\V(\Gamma-v)\sm F\notin
\Gamma-v$ and hence $A\notin \Gamma$ and $A\notin \cc^\vee$. But $\dim (A)=(n-1)-\dim (F)-2= n-d-2= \dim
(\Gamma)=\dim (\cc^\vee)$. Thus $Fv= \V(\Delta)\sm A\in \cc$, that is, $F$ is a facet of $\link_\Delta(v)$.
So the pure $(d-1)$-skeleton of $(\Gamma-v)^\vee$ is contained in $\link_\Delta(v)$. The proof of the reverse
inclusion is similar.

\ref{dual of F2} Noting that $(\F(\Gamma-v))^\vee$ is exactly the set of $(d-1)$-dimensional faces of
$(\Gamma-v)^\vee$, we see that part \ref{dual of F1} indeed states $\cd= (\F(\Gamma-v))^\vee$, which is
equivalent to \ref{dual of F2}.

\ref{dual of F3} Assume that $e\in \MS(\cd)$. Then for $x\in \V(\cd)\sm e$ we have $ex\in \cd \iff exv\in
\F(\Delta)=\cc$. Therefore, $\N_\cd[e]=\N_\cc[ev]\sm \{v\}$. If $ev\in \sms(\cc)$ and $A$ is a
$(d-1)$-dimensional subset of $\N_\cd[e]$, then $Av\se \N_\cc[ev]$ which is a clique and hence $Av\in \cc$,
that is, $A\in \cd$. So $\N_\cd[e]$ is a clique and $e\in \sms(\cc)$.

Conversely, assume that $e\in \sms(\cd)$ and $A$ is a $d$-dimensional subset of $\N_\cc[ev]$. We must show
$A\in \cc$. If $v\in A$, then $A\sm \{v\}\se \N_\cd[e]$ and hence $A\sm\{v\} \in \cd$, which means $A\in
\cc$. Thus suppose $v\notin A$. If $A\notin \cc$, then $v\in B=\V(\cc)\sm A\in \cc^\vee=\F(\Gamma)$.
Consequently, $B\sm\{v\}\in \Gamma-v$ and there is a $F\in \F(\Gamma-v)=\cd^\vee$ containing $B\sm\{v\}$. So
$A'= \V(\cd)\sm F\notin \cd$. Since $\dim(\Gamma-v)=\dim \Gamma= n-d-2= |F|-1$ and $\V(\cd)=\V(\cc)\sm
\{v\}$, we see that $|A'|=d$. Also $A'\se A\se\N_\cc[ev]\sm \{v\}=\N_\cd[e]$, and as $e$ is simplicial, we
get $A'\in \cd$, a contradiction from which the result follows.
\end{proof}

\begin{ex}\label{vdec-ex}
Suppose that $\cc$ is a 2-clutter on $[6]$ with circuits 123, 124, 134, 234, 345, 346, 126 and $\Delta$,
$\Gamma$ and $\cd$ are defined as in \ref{dual of F(link)}. Then
\begin{align}
 \Gamma= \lg \cc^\vee \rg  =&\lg  136, 146, 236, 246, 346,
  135, 235, 145, 245,
  123, 124,  134,  234 \rg \  \hbox{and} \cr
\Gamma-6=& \lg 135, 235, 145, 245,
  123, 124,  134,  234 \rg.
\end{align}
So 6 is a shedding vertex of $\Gamma$. Note that $\F(\Gamma-6)$ is a 3-clutter on $[5]$, thus
$(\Gamma-6)^\vee = \lg 12, 34 \rg$ which is exactly $\link_\Delta(6)$. Also $\cd= \{12,34\}$ is indeed a
graph and $\sms(\cd)=\{1,2,3,4\}$. As claimed in the previous lemma, simplicial maximal subcircuits of $\cc$
which contain the vertex 6 are $16, 26, 36,46$.
\end{ex}

\begin{lem}\label{shedding stable}
Let $\cc$ be a $d$-clutter. Assume that $\Delta=\lg \cc \rg$ and $\Gamma=\lg \cc^\vee \rg$. Let $v$ be a shedding
vertex of $\Gamma$ and $v\in e\in \MS(\cc)$. Then $v$ is a shedding vertex of $\lg (\cc-e)^\vee \rg=
(\Delta(\cc-e))^\vee$. Furthermore, if $\Delta'=\lg \cc-e \rg$, then $\F(\link_{\Delta'} v)= \F(\link_\Delta v)-e'$,
where $e'=e\sm\{v\}$.
\end{lem}
\begin{proof}
Let $\Gamma'=\lg (\cc-e)^\vee \rg$ and suppose that $F\in \F(\link_{\Gamma'} v)$. We have to show that $F$ is not a
facet of $\Gamma'-v$. Note that $\cc^\vee  \se (\cc-e)^\vee$ and hence $\Gamma-v\se \Gamma'-v$. Therefore, it suffices
to show that there is a $G\in \F(\Gamma)=\cc^\vee$ with $v\notin G$ such that $F\su G$. We know that $Fv\in
\F(\Gamma')=(\cc-e)^\vee$, so $\V\sm Fv\notin \cc-e$ where  $\V=\V(\cc-e)=\V(\cc)$. As $v\in e$, $e\not\se \V\sm Fv$
and thus $\V\sm Fv\notin \cc$. It follows that $Fv\in \Gamma$, that is, $F\in \link_\Gamma v$. No face of $\link_\Gamma
v$ is a facet of $\Gamma-v$, because $v$ is a shedding vertex of $\Gamma$. Consequently, $F$ is strictly contained in
the facet $G$ of $\Gamma-v$, as claimed. The proof of the ``furthermore'' statement, which follows from definitions, is
left to the reader.
\end{proof}

\begin{ex}\label{vdec-ex2}
Let $\cc, \Delta, \Gamma$ and $v=6$ be as in Example \oldref{vdec-ex} and set $e=26$, $\Gamma'=\lg
(\cc-e)^\vee \rg$ and $\Delta'=\lg \cc-e\rg$. Then $\cc-e=\cc\sm \{126\}$ and $(\cc-e)^\vee= \cc^\vee \cup
\{345\}$ and the facets of $\Gamma'$ (resp. $\Gamma'-6$) are obtained by adding the face 345 to the set of
facets of $\Gamma$ (resp. $\Gamma-6$). Thus $\Gamma'-6$ is pure and of dimension 2 and hence 6 is still a
shedding vertex in $\Gamma'$. Also the only facet of $\link_\Delta(6)$ which does not contain the vertex $2=
e\sm\{v\}$ is 34 and $\link_{\Delta'}(6)=\lg 34 \rg $.
\end{ex}

Suppose that $v\in\V(\cc)$ has not appeared in any circuit of $\cc$. Then chordality (and many other properties) of
$\cc$ and $\cc-v$ are equivalent, although $\V(\cc)\neq \V(\cc-v)$. In this case, we misuse the notation and write
$\cc=\cc-v$.

\begin{lem}\label{ver del}
Let $\cc$ be a $d$-clutter. Assume that $\Delta=\lg \cc \rg$, $v$ is a shedding vertex of $\Gamma=\lg \cc^\vee \rg$ and
$\cd=\F(\link_\Delta v)$. If $\cd$ is chordal, then there is a sequence $e_1, \ldots, e_t$ with $e_i\in
\sms(\cc_{i-1})$, where $\cc_0=\cc$ and $\cc_i=\cc_{i-1}-e_i$, such that $\cc_t=\cc-v$.
\end{lem}
\begin{proof}
We prove the statement by induction on $|\cc|$. If $\cd=\tohi$, then $\cc=\cc-v$ and the claim holds trivially. So
assume $\cd\neq \tohi$ and $e'\in \sms(\cd)$ be such that $\cd-e$ is chordal. By \ref{dual of F(link)}, $e=e'v\in
\sms(\cc)$. Let $\Delta'=\lg \cc-e \rg$ and $\cd'=\F(\link_{\Delta'}v)$, then according to \ref{shedding stable},
$\cd'=\cd-e'$ is chordal and $v$ is a shedding vertex of $\lg (\cc-e)^\vee \rg$. Thus by applying the induction
hypothesis on $\cc-e$, the assertion follows.
\end{proof}

\begin{ex}\label{vdec-ex3}
Let's use the notations of Examples \oldref{vdec-ex} and \oldref{vdec-ex2}. We saw that $\cd=\{12, 34\}$,
which is a chordal graph with $2\in \sms(\cd)$ and $4\in \sms(\cd-2)$. Now we see that $26\in \sms(\cc)$ and
$46\in \sms(\cc-26)$ and in $\cc-26-46$ the vertex 6 has not appeared in any circuit. In our notations, this
means that $\cc-26-46=\cc-6$, as asserted in the previous lemma.
\end{ex}

\begin{lem}\label{almost complete}
Suppose that $\cc$ is obtained from a complete clutter by deleting exactly one circuit. Then $\cc$ is chordal.
\end{lem}
\begin{proof}
Let $\cc=\cc_0-F$, where $\cc_0$ is a complete clutter and $F\in \cc_0$. By \cite[Corollary 3.11]{chordal}, $\cc_0$ is
chordal and there is a sequence of simplicial maximal subcircuits $e_1,\ldots, e_t$ which sends $\cc_0$ to the empty
clutter. By symmetry of $\cc_0$, $e_1$ can be any maximal subcircuit of $\cc_0$ and we can assume $e_1\se F$. Hence
$\cc-e= \cc_0-e$ is chordal and it is easy to see that $e\in \sms(\cc)$, that is, $\cc$ is also chordal.
\end{proof}

Now we are ready to state the main result of this section.
\begin{thm}\label{vdec main}
Assume that $\cc$ is a $d$-clutter and $\lg \cc^\vee \rg$ is vertex decomposable. Then $\cc$ is chordal.
\end{thm}
\begin{proof}
We use induction on the number of vertices of $\cc$. Let $\Delta=\cc$ and $\Gamma= \lg \cc^\vee \rg$. If
$\Gamma$ is a simplex, then $\cc$ is obtained from a complete clutter by deleting one circuit and by
\ref{almost complete} is chordal. Thus assume that $\Gamma$ is not a simplex. So there is a shedding vertex
$v$ of $\Gamma$ such that both $\Gamma-v$ and $\link_\Gamma v$ are vertex decomposable. Setting
$\cd=\F(\link_\Delta v)$, it follows from \ref{dual of F(link)} that $\lg \cd^\vee \rg= \Gamma-v$ and is
vertex decomposable. Therefore, $\cd$ is chordal by induction hypothesis and by \ref{ver del}, there is a
sequence $e_1, \ldots, e_t$ with $e_i\in \sms(\cc_{i-1})$, where $\cc_0=\cc$ and $\cc_i=\cc_{i-1}-e_i$, such
that $\cc_t=\cc-v$. Consequently, we just need to show that $\cc-v$ is chordal. For this we show that $\lg
(\cc-v)^\vee \rg=\link_\Gamma v$, which is vertex decomposable and the result follows by the induction
hypothesis:
\begin{align}
F\in \F(\link_\Gamma v) & \iff Fv\in \F(\Gamma)=\cc^\vee \iff \V\sm Fv\notin \cc \cr
& \iff (\V\sm \{v\})\sm F\notin \cc-v \iff F\in (\cc-v)^\vee.
\end{align}

\end{proof}

\begin{ex}\label{vdec-ex4}
With notations as in Example \oldref{vdec-ex3}, we saw that $\cc-26-46=\cc-6$ and $26\in \sms(\cc)$ and
$46\in \sms(\cc-26)$. Now in $\lg (\cc-6)^\vee \rg= \lg 13,14,23,24,34 \rg$ the vertices $1,2,3,4$ are
shedding vertices. Applying \ref{ver del} with $\cc-6$ instead of $\cc$ and with $v=1$, we find for example
$12\in \sms(\cc-6)$ and $13\in \sms(\cc-6-12)$ with $\cd=(\cc-6)-12-13= \cc-6-1= \{234,345\}$. Applying
\ref{ver del} again we can find say $23\in \sms(\cd)$ and finally $35\in \sms(\cd-23)$ to get that $\cc$ is
chordal.
\end{ex}

In \cite{w-chordal}, a vertex of a not necessarily uniform clutter, $\cd$ is called a \emph{simplicial vertex} if for
every $e_1,e_2\in \cd$ with $v\in e_1,e_2$, there is an $e_3\in \cd$ such that $e_3\se (e_1\cup e_2)\sm \{v\}$. Also
the \emph{contraction} $\cd/v$ is defined as the clutter of minimal sets of $\{e\sm \{v\}| e\in \cd\}$. Now if every
clutter obtained from $\cd$ by a sequence of deletions or contractions of vertices has a simplicial vertex, then
Woodroofe in \cite{w-chordal} calls $\cd$ chordal and we call it \emph{W-chordal}. In \cite[Corollary 3.7]{chordal} it
is proved that every W-chordal clutter which is uniform  is chordal. As a corollary of the above theorem, we can get
the following slightly stronger version of Corollary 3.7 of \cite{chordal}.

\begin{cor}\label{W-chordal}
Suppose that $\cd$ is a (not necessarily uniform) W-chordal clutter, $d=\min \l\{|F|\big| F\in \cd \r\}$ and $\cc=\l\{
F\in \cd \big| |F|=d \r\}$. Then $\cc$ is chordal.
\end{cor}
\begin{proof}
Theorem 6.9 of \cite{w-chordal} states that the Alexander dual of the independence complex of $\b\cc$ (for the
definition of independence complex, see \cite[p. 3]{w-chordal}) is vertex decomposable. But independence complex of
$\b\cc$ is exactly $\Delta(\cc)$ and hence the result follows from Theorem \oldref{vdec main}.
\end{proof}

%%%%%%%%%%%%%%%%%%%%%%%%%%%%%%%%%%%%%%%%%%%%%%%%%%%%%%%%%%%%%%%%%%%%%%%%%%%%%%%%%%%%%%%%%%%%%%%%%%%%%%%%%%%%%%%%%%
                                     \section{Chordality and ascent of clutters}
%%%%%%%%%%%%%%%%%%%%%%%%%%%%%%%%%%%%%%%%%%%%%%%%%%%%%%%%%%%%%%%%%%%%%%%%%%%%%%%%%%%%%%%%%%%%%%%%%%%%%%%%%%%%%%%%%%

In \cite{CF1,CF2}, the concept of chorded simplicial complexes is defined and it is proved that $I_\Delta$ has a
$(d+1)$-linear resolution over a field of characteristic 2, \ifof  $\Delta$ is chorded and it is the clique complex of
a $d$-clutter. We briefly recall this concept. Let $\Delta$ be a simplicial complex. We say that $\Delta$ is
\emph{$d$-path connected}, when it is pure and for each pair $F$ and $G$ of $d$-faces of $\Delta$ there is a sequence
$F_0,\ldots, F_k$ of $d$-faces of $\Delta$, with $F_0=F$, $F_k=G$ and $|F_i\cap F_{i-1}|=d$. Also a pure
$d$-dimensional simplicial complex $\Delta$ is called a \emph{$d$-cycle}, when $\Delta$ is $d$-path connected and every
$(d-1)$-face of $\Delta$ is contained in an even number of $d$-faces of $\Delta$. A $d$-cycle $\Delta$ is called
\emph{face-minimal}, if there is no $d$-cycle on a strict subset of the $d$-faces of $\Delta$. Finally, we say that a
pure $d$-dimensional simplicial complex $\Delta$ is \emph{$d$-chorded}, when for every face-minimal $d$-cycle
$\Omega\se \Delta$ which is not $d$-complete, there exists a family $\{\Omega_1, \ldots,\Omega_k\}$ of $2\leq k$
$d$-cycles with each $\Omega_i\se \Delta$ such that
\def\theenumi{\alph{enumi}}
\begin{enumerate}
\item \label{a} $\Omega\se \cup_{i=1}^k \Omega_i$,
\item \label{b} each $d$-face of $\Omega$ is contained in an odd number of $\Omega_i$'s,
\item \label{c} each $d$-face in $\cup_{i=1}^k \Omega_i\sm \Omega$ is contained in an even number of $\Omega_i$'s,
\item \label{d} $\V(\Omega_i)\su \V(\Omega)$.
\end{enumerate}
\def\theenumi{\roman{enumi}}
We refer the reader to \cite[Section 3]{CF1}, for several examples of these concepts.

Suppose that $\Delta$ is the clique complex of a $d$-clutter. Connon and Faridi, first in \cite{CF1} prove that if
$I_\Delta$ has a $(d+1)$-linear resolution over a field of characteristic 2, then $\Delta$ is $d$-chorded. Then in
\cite[Theorem 18]{CF2}, the same authors show that $I_\Delta$ has a $(d+1)$-linear resolution over a field of
characteristic 2, \ifof $\Delta$ is $d$-chorded, and moreover $\Delta^{[m]}$ is $m$-chorded for all $m>d$
($\Delta^{[m]}$ is the pure $m$-skeleton of $\Delta$). Now if $I_\Delta=I(\b \cc)$ for a $d$-clutter $\cc$, then
$\Delta=\Delta(\cc)$ and hence $\F(\Delta^{[m]})$ is just the set of cliques of $\cc$ with dimension $m$. Therefore,
the aforementioned results show that to study when $I(\b\cc)$ has a linear resolution, it may be useful to consider the
higher dimensional clutters $\F(\Delta(\cc)^{[m]})$. Motivated by this observation, we define the ascent of a clutter
as follows.

\begin{defn}\label{def ascent}
Let $\cc$ be a $d$-clutter. We call the family of all  $(d+1)$-dimensional cliques of $\cc$ (that is, cliques on $d+2$
vertices), the ascent of $\cc$ and denote it by $\cc^+$. In other words, $\cc^+=\F \p{ \Delta(\cc)^{[d+1]} }$.
\end{defn}

\begin{ex}
Let $\cc$ be as in Example \oldref{counter Ex}. Then $\cd=\MS(\cc)$ is a 1-clutter (that is, a graph) on
$[8]$. The circuits of $\cd$ are the edges in \oldref{Fig-Dunce}. Now $\cd^+$ is the set of cliques of size 3
of $\cd$, in particular, all of the triangles in \oldref{Fig-Dunce} are in $\cd^+$. Thus $\cc \se \cd^+$.
Note that all of the edges 12, 13, 23, 24, 34 are in $\cd$, so $123$ and $234$ are in $\cd^+$ and $\cd^+\neq
\cc$. This shows that $\cc$ is not of the form $(\cc')^+$ for some 1-clutter $\cc'$ (else $\cc'$ should
contain $\MS(\cc)=\cd$ and $\cd^+\se \cc'^+$, a contradiction). Now $\cc^+$ means the set of all 4-subsets of
$[8]$ such as $A$, with the property that all $3$-subsets of $A$ are in $\cc$. But there is no such
$4$-subset of $[8]$ and hence $\cc^+=\tohi$. Moreover, one could check that for example
\begin{align}
 \cd^+\supsetneq & \cc \cup \{123, 23a,12a,13a|a=4,5,7,8\} \cr
 \cd^{++}= & (\cd^+)^+ \supsetneq \{123a, 136a|a=4,5,7,8\} \cr
 \cd^{+++}= & \{123ab, 136ab|ab=45,78,48\},\quad \cd^{++++}=\tohi
\end{align}
\end{ex}
Our first result, considers how the property of having a linear resolution behaves under ascension. Recall
that throughout the paper, we assume that $\cc$ is a $d$-clutter on vertex set $V$ with $|V|=n$ and $K$ is an
arbitrary field.

\begin{prop}\label{lin res +}
Let $\Delta=\Delta(\cc)$. Then $I(\b\cc)$ has a linear resolution over $K$, \ifof $I(\b{\cc^+})$ has a linear
resolution over $K$ and $\tl H_d(\Delta_W;K)=0$ for all $W\se V$.
\end{prop}
\begin{proof}
Note that if $\Delta'= \Delta(\cc^+)$, then $\Delta'_W$ and $\Delta_W$ differ only in $d$-dimensional faces
and hence $\tl H_t(\Delta'_W; K)= \tl H_t(\Delta_W; K)$ for each $t> d$ and $W\se V$. Thus the result follows
directly form Fr\"oberg's theorem \ref{Fro}. Note that $\Delta^{[d]}= \lg \cc \rg$, thus $\Delta$ is
$d$-chorded \ifof $\lg \cc \rg$ is so.
\end{proof}

The previous simple result shows why the concept of ascent of a clutter can be useful. For example, we show
that a main theorem of \cite{CF2} can simply follow \ref{lin res +} and the results of \cite{CF1}. First we
state the needed results of \cite{CF1} as a lemma.

\begin{lem}\label{CF1-main}
Let $\cc$ be a $d$-clutter. Suppose that $K$ is a field of characteristic 2 and $\Delta=\Delta(\cc)$. Then $\tl
H_d(\Delta_W;K)=0$ for all $W\se V$ \ifof $\lg \cc \rg$ is $d$-chorded.
\end{lem}
\begin{proof}
(\rgive) is \cite[Proposition 5.8]{CF1} and the proof for (\give) is the proof of part 2 of \cite[Theorem
6.1]{CF1}.
\end{proof}

We inductively define a \emph{CF-chordal} clutter. We consider $\tohi$, CF-chordal and we say that a non empty
$d$-clutter $\cc$ is CF-chordal, when $\lg \cc \rg$ is $d$-chorded and $\cc^+$ is CF-chordal. It is easy to check that
$\cc$ is CF-chordal \ifof $\Delta(\cc)^{[m]}$ is $m$-chorded for all $m\geq d$. In \cite{CF2} simplicial complexes with
this property are called \emph{chorded}. Thus the next result is indeed Theorem 18 of \cite{CF2}, which is one of the
two main theorems of that paper. Before stating this result, it should be mentioned that an example of a clutter $\cc$
with $\lg \cc \rg$ $d$-chorded but $\lg \cc^+ \rg$ not $(d+1)$-chorded, is presented in \cite[Example 16]{CF2}. This
example shows that $\lg \cc \rg$ can be $d$-chorded while $\cc$ is not CF-chordal.

\begin{cor}[{\cite[Theorem 18]{CF2}}]\label{CF-2 main}
Let $\cc$ be a $d$-clutter. Suppose that $K$ is a field of characteristic 2. Then $I(\bar \cc)$ has a linear resolution
over $K$, \ifof $\cc$ is CF-chordal.
\end{cor}
\begin{proof}
This follows from \ref{CF1-main} and \ref{lin res +} and a simple induction on $n-d$.
\end{proof}

As the following result shows, similar to having a linear resolution, chordality is also preserved under ascension.
\begin{thm}\label{chordal +}
Let $\cc$ be a $d$-clutter. If $\cc$ is chordal, then so is $\cc^+$.
\end{thm}
\begin{proof}
We use induction on $|\cc|$. The result is clear for $|\cc|=0,1$, where $\cc^+=\tohi$. Choose $e\in \sms(\cc)$ such
that $\cc-e$ is chordal. If $e$ is contained in only one circuit, then it is not contained in any clique on more than
$d+1$ vertices. Thus $\cc^+=(\cc-e)^+$ and the claim follows the induction hypothesis. So we assume that $\N_\cc[e]$ is
a clique on at least $d+2$ vertices. Let $F$ be any $d+1$ subset of $\N_\cc[e]$ with $e\se F$. If $v\in
\N_{\cc^+}[F]\sm F$, then $Fv$ is a clique in $\cc$ and hence $ev\in \cc$. This shows that $\N_{\cc^+}[F]\se
\N_\cc[e]$. Let $\{v_1, \ldots, v_t\}= \N_\cc[e]\sm e$, $F_i=ev_i$ and $\cc_i= \cc- F_1- \cdots- F_{i}$. Note that
$\cc_i^+= \cc^+ -F_1- \cdots- F_i$. We show that $F_i\in \sms(\cc^+_{i-1})$, if $F_i\in \MS(\cc_{i-1}^+)$. Then as
$\cc_t=\cc-e$ is chordal it follows by the induction hypothesis that $\cc_t^+$ and hence $\cc^+$ are chordal.

Suppose that $G\se \N_{\cc_{i-1}^+}[F_i]$ and $|G|=d+2$. Then by the above argument $G\se \N_\cc[e]$ and hence $G\in
\cc^+$. If $v_j\in G$ for some $j<i$, then $F_iv_j\in \cc_{i-1}^+$. But $F_j= ev_j\se F_iv_j$ and $F_j\notin
\cc_{i-1}$, a contradiction. Therefore, no $v_j\in G$ for $j<i$ and hence no $F_j\se G$ for such a $j$. Thus $G$ is a
clique in $\cc_{i-1}$, that is, $G\in \cc_{i-1}^+$ and hence $F_i\in \sms(\cc_{i-1}^+)$.
\end{proof}

\begin{ex}\label{chorded octa}
Suppose that $\cc$ is the 2-clutter shown in \oldref{Fig octa}. The circuits of $\cc$ are the faces of the
hollow octahedron and the four triangles shown in a darker color. It is not hard to check that $\cc$ is
chordal, for example, the sequence 12, 14, 64, 62, 13, 23, 63, 43, is a sequence of consecutive simplicial
maximal subcircuits, by deletion of which, we reach to the empty clutter. Hence $\cc^+=\{1235,1345, 2356,
3456\}$ is also chordal by \ref{chordal +}. Indeed, all maximal subcircuits of $\cc^+$ are simplicial except
for the four darker ones.

\begin{figure}[!ht]
  \centering
  \includegraphics[scale=1]{{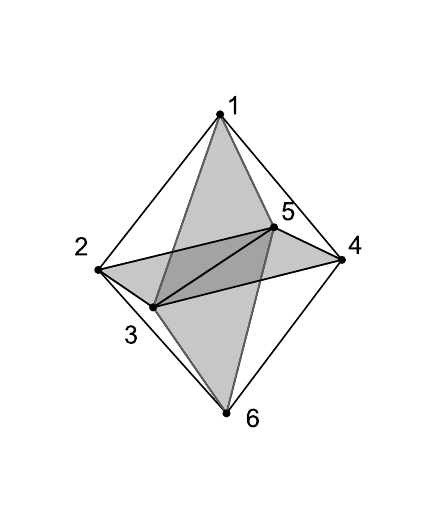}}
  \caption{A chordal 2-clutter}\label{Fig octa}
\end{figure}
\end{ex}

Next we state a theorem that shows some connections between deleting elements of $\sms(\cc^+)$ and having a linear
resolution.
\begin{thm}\label{lin res sms}
Let $\cc$ be a $d$-clutter. Suppose that $\cc^+\neq \tohi$. Then, considering the following statements, we have
\ref{lin res sms 1} \give\ \ref{lin res sms 2} \give\ \ref{lin res sms 3}. Moreover, if $\sms(\cc^+)\neq \tohi$, then
(\ref{lin res sms 1}, \ref{lin res sms 2} and \ref{lin res sms 3} are equivalent.
\begin{enumerate}
  \item \label{lin res sms 1} There is a $F\in \MS(\cc^+)$ such that $I(\b{\cc-F})$ has a linear resolution
      over $K$ and also $I(\b{\cc^+})$ and $I(\b{\cc-v})$ have linear resolutions over $K$ for each $v\in
      V$.
  \item \label{lin res sms 2} $I(\b \cc)$ has a linear resolution over $K$.
  \item \label{lin res sms 3} For all $F\in \sms(\cc^+)$, $I(\b{\cc-F})$ has a linear resolution over $K$ and also
      $I(\b{\cc^+})$ and $I(\b{\cc-v})$ have linear resolutions over $K$ for each $v\in V$.
\end{enumerate}
\end{thm}
\begin{proof}
\ref{lin res sms 1} \give\ \ref{lin res sms 2}: Set $\Delta=\Delta(\cc)$. According to \ref{lin res +}, we have to show
that $\tl H_d(\Delta_W; K)=0$ for all $W\se V$. If $W\su V$, say $v\in V\sm W$, then $\Delta_W= \Delta(\cc_W)=
\Delta((\cc-v)_W)$. Thus by Fr\"oberg's theorem, as $I(\b{\cc-v})$ has a linear resolution over $K$, it follows that
$\tl H_d(\Delta_W; K)=0$. Consequently it remains to show that $\tl H_d(\Delta; K)=0$.

Assume that $x= \sum_{i=1}^l a_iF_i\in \tl C_d(\Delta; K)$ with $\rnd_d(x)=0$ where $0\neq a_i\in K$ and $F_i\in \cc$.
We have to show that $x\in \Im \rnd_{d+1}$. Let $F\in \MS(\cc^+)$ be such that $I(\b{\cc-F})$ has a linear resolution
over $K$. Assume that for some $i$, we have $F_i=F$. As $F\in \MS(\cc^+)$, there is a $G\in \cc^+$ such that $F\se G$.
Now $y= \rnd_{d+1}(a_iG)$ has a term $\pm a_iF_i$. Hence in one of the two elements $x\mp y$, the coefficient of $F$ is
zero. Call this element $z$. Then $\rnd_d(z)=\rnd_d(x)\mp \rnd_d(\rnd_{d+1}(a_iG))= 0$. Thus we can assume that for no
$i$, $F_i=F$. Then $x\in \tl C_d(\Delta(\cc-F); K)$ and as $I(\cc-F)$ has a linear resolution, $\tl H_d(\Delta(\cc-F);
K)=0$. So there are $G_i$'s in $(\cc-F)^+\se \cc^+$ and $b_j$'s in $K$ such that $x=\rnd_{d+1}(\sum b_jG_j)\in \Im
\rnd_{d+1}$.

\ref{lin res sms 2} \give\ \ref{lin res sms 3}: As $\Delta(\cc-v)_W = \Delta(\cc)_{W\sm\{v\}}$ and by Fr\"oberg's
theorem, $I(\b{\cc-v})$ has a linear resolution over $K$ for each $v\in V$. According to \ref{lin res +},
$I(\b{\cc^+})$ also has a linear resolution over $K$. Now let $F\in \sms(\cc^+)$ and also assume that $V=[n]$. Let
$L=\lg x_F \rg$ and $I=I(\b\cc)$. We have to show $I(\b{\cc-F})=I+L$ has a linear resolution over $K$.

First we compute the minimal generating set of $I\cap L$. Suppose that $u$ is one of the minimal generators of the
squarefree monomial ideal $I\cap L$. Then as $x_F|u$, we should have $u=x_{F\cup A}$ for some $A\se [n]\sm F$. Since
$u\in I$, there is a $G_0\in \b\cc$ such that $G_0\se F\cup A$. Assume $|A|>1$ and $v\in A$. Then for no $G\in \b\cc$
we have $G\se Fv$, else $x_{Fv}\in I\cap L$ which contradicts $u$ being a minimal generator. This means that $Fv$ is a
clique of $\cc$ and hence $v\in \N_{\cc^+}[F]$. Since $v\in A$ was arbitrary, we have $F\cup A\se N_{\cc^+}[F]$ which
is a clique, because $F$ is simplicial. But this contradicts $G_0\se F\cup A$ and it follows that $|A|=1$, say
$A=\{a\}$. As $G_0\se Fa$, we get $a\notin \N_{\cc^+}[F]$. On the other hand, for an arbitrary $a\in B=[n]\sm
\N_{\cc^+}[F]$, it is easy to see that $x_{Fa}\in I\cap L$. Consequently, $I\cap L=\lg x_{Fa}| a\in B\rg =x_F  \lg x_a|
a\in B \rg$.

Since multiplying in $x_F$ is an $S$-isomorphism of degree $d+1$ from $\lg x_a| a\in B \rg$ to $I\cap L$ and by
\cite[Corollary 7.4.2]{hibi}, $I\cap L$ has a $(d+2)$-linear resolution over $K$. Also $I$ and $L$ have $(d+1)$-linear
resolutions over $K$. Now consider the exact sequence
$$0\to I\cap L\to I\dis L \to I+L \to 0.$$
Writing the $(i+j)$-th degree part of the long exact sequence of $\Tor^S(K,-)$ applied on the above sequence, we get
the exact sequence
$$ \cdots \to \Tor_i^S(K, I\dis L)_{i+j} \to \Tor_i^S(K, I+L)_{i+j} \to \Tor_{i-1}^S(K, I\cap L)_{i-1+(j+1)} \to \cdots.$$
If $j\neq d+1$ then both flanking terms are zero and hence the middle term is also zero. This means that $I+L$ has a
$(d+1)$-linear resolution, as required.
\end{proof}

In the following example, we show how the above theorem can be used.
\begin{ex}
Suppose that $\cc$ is as in Example \oldref{counter Ex}. Let $\cd=\cc\cup \{278\}$. Then $\cd^+=\{1278\}$,
$278\in \sms(\cd^+)$ and $\cd^+$ is chordal. Hence $I(\b{\cd^+})$ has a linear resolution over every field.
Also it is easy to verify that for any $W\su [8]$, $\cd_W$ is chordal (indeed, if $W\neq \{1,2,7,8\}$, then
either $\cd_W$ is empty or it has an edge which is contained in exactly one triangle and hence is simplicial
and if $W= \{1,2,7,8\}$, $\cd_W$ is complete) and thus has a linear resolution over every field. Thus, as
$I(\b{\cd-278})= I(\b \cc)$ has a linear resolution over any field (see Example \oldref{counter Ex}) and
according to \ref{lin res sms}, we get that $I(\b \cd)$ has a linear resolution over every field.
\end{ex}

Note that  in \ref{lin res sms}\ref{lin res sms 3}, against part \ref{lin res sms 1}, we cannot replace $F\in
\sms(\cc^+)$ with $F\in \MS(\cc^+)$. For example, if $\cc$ is as in Example \oldref{chorded octa}, then $\lg
\cc-135 \rg$ is not 2-chorded, and hence $I(\b{\cc-135})$ has not a linear resolution over $\z_2$, although
$\cc$ is chordal and $I(\b\cc)$ has a linear resolution over every field.

Also note that \ref{lin res sms}\ref{lin res sms 3} does not necessarily imply \ref{lin res sms}\ref{lin res
sms 1}, since we may have a clutter $\cc$ satisfying \ref{lin res sms}\ref{lin res sms 3}, but with
$\sms(\cc^+)=\tohi$. Although, it should be mentioned that the author could not find such an example. Thus we
raise the following question.
\begin{ques}
Are the three statements of \ref{lin res sms} equivalent for a $d$-clutter $\cc$ with a non-empty ascent?
\end{ques}

\begin{rem}
In the proof of \ref{lin res sms 1} \give\ \ref{lin res sms 2} of \ref{lin res sms}, we proved that if there
exists $F\in \MS(\cc^+)$ such that $I(\b{\cc-F})$ has a linear resolution over $K$, then $\tl
H_d(\Delta(\cc); K)=0$.
\end{rem}
The following is also a corollary to the proof of \ref{lin res sms}. Recall that $\beta_i(I)=\sum_{j} \beta_{ij}(I)$.
\begin{cor}
Let $\cc$ be a $d$-clutter. Suppose that $I= I(\b \cc)$ has a linear resolution over $K$, $F\in \sms(\cc^+)$ and
$J=I(\b{\cc-F})$. Then $J$ has a linear resolution over $K$ with Betti numbers $\beta_i(J)= \beta_i(I)+{t \choose i}$,
where $t=n-|\N_{\cc^+}[F]|$.
\end{cor}
\begin{proof}
Suppose that $L=\lg x_F\rg$ and $B=[n]\sm \N_{\cc^+}[F]$ as in the proof of \ref{lin res sms 2} \give\ \ref{lin res sms
3}. So $t=|B|$ and if $L'= \lg x_a|a\in B \rg$, then $I\cap L=x_F L'$ and $\beta_i(I\cap L)= \beta_i(L')$ which by
\cite[Corollary 7.4.2]{hibi} is equal to $\sum_{k=1}^t {k-1 \choose i}= \sum_{k=i}^{t-1} {k \choose i}= {t \choose i}$
. Also $\beta_0(L)=1$ and $\beta_i(L)=0$ for all $i>0$. Thus by taking $\dim_K$ of the long exact sequence of
$\Tor_i^S(K,-)$ in the proof of \ref{lin res sms}, we get $\beta_i(J)= \beta_i(I+L)= \beta_i(I)+\beta_i(L)+
\beta_{i-1}(I\cap L)$, where the last term is assumed to be zero if $i=0$.
\end{proof}

If in \ref{lin res sms} we restrict to the case that $\textrm{char} K=2$ and by using \ref{CF-2 main}, we get
the following.
\begin{cor}\label{lin res sms char 2}
Let $\cc$ be a $d$-clutter. Suppose that $\cc^+\neq \tohi$.
\begin{enumerate}
\item \label{R2 1} If $\cc^+$ is CF-chordal and there is a $F\in \MS(\cc^+)$ such that $\cc-F$ is CF-chordal and
    $\cc-v$ is CF-chordal for all $v\in \V(\cc)$, then $\cc$ is CF-chordal.
\item \label{R2 2} If $\cc$ is CF-chordal, then for all $F\in \sms(\cc^+)$ and all $v\in \V(\cc)$, $\cc-F$ and
    $\cc-v$ are CF-chordal.
\end{enumerate}
Moreover, if $\sms(\cc^+)\neq \tohi$, then the converses of both \ref{R2 1} and \ref{R2 2} hold.
\end{cor}

In the rest of the paper, we study if we can get some results similar to \ref{lin res sms} for having linear quotients
or being chordal instead of having a linear resolution. For having linear quotients we have:
\begin{thm}\label{lin q +}
Let $\cc$ be a $d$-clutter. Suppose that $I(\b\cc)$ has linear quotients. Then $I(\b{\cc^+})$, $I(\b{\cc-F})$  and
$I(\b{\cc-v})$ have linear quotients for each $F\in \sms(\cc^+)$ and $v\in \V(\cc)$.
\end{thm}
\begin{proof}
If we denote the ideal generated by all squarefree monomials in $I$ with degree $t$ by $I_{[t]}$, then it is easy to
see that $I(\b{\cc^+})= I(\b\cc)_{[d+2]}$. Hence the fact that $I(\b{\cc^+})$ has linear quotients follows from
\cite[Corollary 2.11]{solyman}. But for the convenience of the readers, we present a direct shorter proof for this.
Assume that $x_{F_1}, \ldots, x_{F_t}$ is an admissible order of $I(\b\cc)$.  Then $\b{\cc^+}=\{F_iv|v\in V\sm F_i,
1\leq i\leq t\}$. For each $1\leq i\leq t$, let $\cc_i= \{F_iv|v\in V\sm F_i\} \sm (\cup_{j<i} \cc_j)$ and suppose that
$\cc_i=\{F_{i1}, \ldots, F_{ir_i}\}$. We prove that $F_{11}, \ldots, F_{1{r_1}}, \ldots, F_{t1}, \ldots, F_{t{r_t}}$
corresponds to an admissible order for $I(\b{\cc^+})$.

Consider two circuits $F_{ij}, F_{i'j'}$ of $\b{\cc^+}$ with $i\leq i'$. If $i=i'$, then $|F_{ij}\sm F_{i'j'}|=1$ and
hence the admissibility condition holds trivially for them. Thus we assume $i<i'$. Therefore by assumption there is a
$l\in F_i\sm F_{i'}$ and a $k< i'$ such that $F_k\sm F_{i'}= \{l\}$. Let $F_{i'j'}\sm F_{i'}=\{v\}$. Note that as
$F_{i'j'}\in \cc_{i'}$, there is no $j<i'$ with $F_{i'j'}\in \cc_j$ by the definition of $\cc_i$'s. So $v\neq l$, else
$F_{i'j'}= F_k\cup (F_{i'} \sm F_k)$ and as $|F_{i'}\sm F_k|=1$, we have $F_{i'j'} \in \cc_j$ for some $j\leq k$, a
contradiction. Consequently, $v\notin F_k$ and $F_kv$ which is a circuit of $\b{\cc^+}$, should appear in some $\cc_j$
with $j\leq k$. Noting that $F_kv\sm F_{i'j'}= \{l\}$, the proof is concluded.

Now assume that $F\in \sms(\cc^+)$. We show that if we add $x_F$ to an admissible order of $I(\b\cc)$, we get an
admissible order of $I(\b{\cc-F})= I(\b\cc)+ \lg x_F \rg$. We just need to show if $G\in \b\cc$, then there is a $G'\in
\b\cc$ and an $l\in G$ such that $G\sm F=\{l\}$. Since $\N_{\cc^+}[F]$ is a clique of $\cc$, $G\not\se \N_{\cc^+}[F]$,
say $l\in G\sm \N_{\cc^+}[F]$. Then $Fl\notin \cc^+$ and hence there is a $G'\se Fl$ with $G'\in \b\cc$. So $G'\sm
F=\{l\}$, as required.

Finally, note that if $x_{F_1}, \ldots, x_{F_t}$ is an admissible order for $I(\b\cc)$, then by deleting $x_{F_i}$'s
with $v\in F_i$, we get an admissible order for $I(\b{\cc-v})$ and $I(\b{\cc-v})$ has linear quotients.
\end{proof}

\begin{ex}
Let $\cc$ be the clutter of Example \oldref{chorded octa}. It is not hard to check that the following is an admissible
order for $I(\b\cc)$: 162, 163, 164, 165, 124, 624, 245, 234. Thus by the previous result, $I(\b{\cc^+})$ has linear
quotients. Indeed, by the proof of \ref{lin q +}, we get the following admissible order for $I(\b{\cc^+})$: 1623, 1624,
1625, 1634, 1635, 1645, 1243, 1245, 6243, 6245, 2453. Also note that the assumption  that $F$ is simplicial, is crucial
in \ref{lin q +}. For example, $I(\b{\cc-135})$ has not a linear resolution and hence has not linear quotients.
\end{ex}

We do not know whether the converse of the above theorem is correct or if a statement similar to \ref{lin res sms 1}
\give\ \ref{lin res sms 2} of \ref{lin res sms} holds for having linear quotients.

Next consider the ``chordal version'' of \ref{lin res sms}, that is, consider the statements obtained by replacing
``$I(\b\cd)$ has a linear resolution over K'' with ``$\cd$ is chordal'' in the three parts of \ref{lin res sms}, where
the replacement occurs for all clutters $\cd$ appearing in these assertions. Then clearly \ref{lin res sms 3} \give\
\ref{lin res sms 1} holds for the chordal version. Also \ref{chordal +} shows that part of \ref{lin res sms 2} \give\
\ref{lin res sms 3} is true for chordality. We will also show that if $\cc$ is chordal and $F\in \sms(\cc^+)$, then
$\lg \cc-F \rg$ is $d$-chorded, which is weaker than being chordal. But before proving this, we utilize \ref{lin q +}
to show that if the chordal version of \ref{lin res sms 3} \give\ \ref{lin res sms 2} (or \ref{lin res sms 1} \give\
\ref{lin res sms 2}) holds, then we can reduce proving statement \ref{B} to verifying it only for clutters with empty
ascent.

In what follows, by a \emph{free maximal subcircuit} of $\cc$, we mean a maximal subcircuit which is
contained in exactly one circuit of $\cc$. In the particular case that $\dim \cc=1$, that is $\cc$ is a
graph, free maximal subcircuits are exactly leaves (or free vertices) of the graph. Also a simplicial complex
$\Delta$ is called \emph{extendably shellable}, if any shelling of a subcomplex of $\Delta$ could be
continued to a shelling of $\Delta$. To see a brief literature review of this concept and some related
results consult \cite{strong chordal, bjorner}.
\begin{cor}\label{main}
Consider the following statements and also statement \ref{B} of Section 2 on a uniform clutter $\cc$.
\def\theenumi{\Alph{enumi}}
\begin{enumerate}
\setcounter{enumi}{1}
\item \label{C} If $\tohi\neq \cc^+$ is chordal and $\cc-F$ and $\cc-v$ are chordal for each $F\in \sms(\cc^+)$ and
    $v\in \V(\cc)$, then $\sms(\cc)\neq \tohi$.
%\item \label{D} If $\cc^+=\tohi$ and $I(\b \cc)$ has a linear resolution over every field, then $\cc$ has a free
%    maximal subcircuit.
%
%\setcounter{enumi}{3} \def\theenumi{\Alph{enumi}$'$}
\item \label{D'} If $\cc^+=\tohi$ and $I(\b \cc)$ has linear quotients, then $\cc$ has a free maximal subcircuit.

\item \label{E} Simon's Conjecture (\cite[Conjecture 4.2.1]{simon}): Every $d$-skeleton of a simplex is extendably
    shellable.
\end{enumerate}
\def\theenumi{\roman{enumi}}
Then: \ref{C} + \ref{D'} \give\ \ref{B} \give\ \ref{D'} + \ref{E}.
\end{cor}
\begin{proof}
Assume that \ref{C} and \ref{D'} hold. We claim that if $I(\b\cc)$ has linear quotients, then $\sms(\cc)\neq \tohi$.
Then it follows from \cite[Theorem 2.1]{our chordal} that \ref{B} is correct. To prove the claim, we use induction on
$(n-d, |\cc|)$ considered with lexicographical order. If $\cc^+=\tohi$, then as every free maximal subcircuit is
simplicial, the claim holds by \ref{D'}. If $\cc^+\neq \tohi$, then applying \ref{lin q +} and using the induction
hypothesis, we see that $\cc^+$, $\cc-F$ and $\cc-v$ are chordal for every $F\in \sms(\cc^+)$ and $v\in \V(\cc)$.
Consequently, the claim follows from \ref{C}.

Now suppose that statement \ref{B} is correct. If $\cc^+=\tohi$, then $|\N_\cc[e]|=d+1$  for every $e\in\sms(\cc)$ and
hence every simplicial maximal subcircuit of $\cc$ is a free maximal subcircuit. Therefore, \ref{D'} holds as a special
case of \ref{B}.  Finally \ref{B} \give\ \ref{E} follows from Theorem 2.5 and Corollary 3.7 of \cite{strong chordal}.
\end{proof}
%Note that \ref{C} is a purely combinatorial statement and the algebraic part of \ref{A} is taken into consideration
%only in \ref{D} with the extra (strong) assumption that $\cc$ has no cliques on more than $d+1$ vertices.

The ``chordal version'' of \ref{lin res sms 2} \give\ \ref{lin res sms 3} in \ref{lin res sms}, partly states that if
$\cc$ is chordal, then for every $F\in \sms(\cc^+)$, $\cc-F$ should be chordal. The author could neither prove nor
reject this in the general case, but we prove a weaker result. In particular, we will show in the case that $\dim
\cc=1$, this is true.
%
%So in the case that $\dim \cc=1$, $\cc$ is 1-chorded \ifof it is a chordal graph. Moreover, if $\cc$ is
%chordal, then by \cite[Theorem 6.1]{CF1}, $\cc$ is $d$-chorded.
%
We need the following lemmas which present an equivalent condition for being $d$-chorded.

\begin{lem}\label{union of face min}
The set of facets of a $d$-cycle is a disjoint union of the set of facets of a family of face-minimal $d$-cycles.
\end{lem}
\begin{proof}
Suppose that $\Omega$ is a $d$-cycle. We use induction on $|\F(\Omega)|$. If $\Omega$ is face-minimal, then the result
is trivial. Thus we can assume that there is a $d$-cycle $\Omega'$ whose facets form a strict subset of $\F(\Omega)$.
Let $\Gamma= \lg \F(\Omega)\sm \F(\Omega') \rg$. Note that each $(d-1)$-face of $\Gamma$ is contained in an even number
of facets of $\Omega$ and an even number of facets of $\Omega'$. Hence each $(d-1)$-face of $\Gamma$ is contained in an
even number of facets of $\Gamma$. Suppose that $\Gamma_1, \ldots, \Gamma_t$ are $d$-path connected components of
$\Gamma$, that is, the maximal pure $d$-dimensional subcomplexes of $\Gamma$ which are $d$-path connected. Note that if
$i\neq j$, then $\Gamma_i$ and $\Gamma_j$ do not share a $(d-1)$-face, else $\Gamma_i \cup \Gamma_j$ is $d$-path
connected, contradicting maximality of both $\Gamma_i$ and $\Gamma_j$. Consequently, for every $i$, each $(d-1)$-face
of $\Gamma_i$ is contained in an even number of $d$-faces of $\Gamma_i$ and $\Gamma_i$ is a $d$-cycle. Noting that
$\F(\Gamma_i)$ are mutually disjoint, and by applying induction hypothesis on $\Gamma_i$'s, it follows that
$\F(\Gamma)$ is a disjoint union of facets of a family $\ca$ of face-minimal $d$-cycles. So $\F(\Omega)$ is the
disjoint union of facets of the family $\ca \cup\{\Omega'\}$ of face-minimal $d$-cycles, as required.
\end{proof}

\begin{lem} \label{d-chorded}
Let $\cc$ be a $d$-clutter. The simplicial complex $\lg \cc \rg$ is $d$-chorded \ifof the facets of every $d$-cycle
$\Omega$ of $\lg \cc \rg$ is the symmetric difference of a family of complete subclutters of $\cc$, each on a
$(d+2)$-subset of $\V(\Omega)$.
\end{lem}
\begin{proof}
(\give): By induction on $|\V(\Omega)|$. Because of \ref{union of face min}, we can assume that $\Omega$ is
face-minimal. If $\Omega$ is $d$-complete, then by face-minimality $|\V(\Omega)|=d+2$ and we are done. Thus we assume
that $\Omega$ is not $d$-complete. Suppose that $\Omega_i$'s are $d$-cycles in $\lg \cc \rg$ satisfying
\ref{a}--\ref{d} in the definition of a $d$-chorded complex. Let $\cd_i=\F(\Omega_i)$. Since each $d$-face in
$\cup_{i=1}^k \cd_i\sm \F(\Omega)$ appears in an even number of $\cd_i$'s by \ref{c}, such $d$-faces are not in the
symmetric difference of $\cd_i$'s. Also by \ref{a} and \ref{b}, we see that each element of $\F(\Omega_i)$ is in an odd
number of $\cd_i$'s and hence is in their symmetric difference. Therefore, $\F(\Omega)=\cd_1 \oplus \cdots \oplus
\cd_k$, where $\oplus$ denotes symmetric difference. Also by \ref{d} of the definition of a $d$-chorded complex, each
$\lg \cd_i \rg$ is a cycle on an smaller number of vertices in $\V(\Omega)$. Thus by applying the induction hypothesis
on $\lg \cd_i \rg$'s we get a decomposition of $\Omega$ as the symmetric difference of a set of complete subclutters of
$\cc$ on $(d+2)$-subsets of $\V(\Omega)$.

(\rgive): Suppose that $\Omega$ is a face-minimal cycle of $\lg \cc \rg$ which is not $d$-complete. Then
$\F(\Omega)= \cd_1 \oplus \cdots \oplus \cd_t$ where $\cd_i$'s are complete subclutters of $\cc$ with size
$d+2$ and $\V(\cd_i)\se \V(\Omega)$. If we set $\Omega_i=\lg \cd_i \rg$, then $\Omega_i$'s clearly satisfy
\ref{a}--\ref{d} of the definition of a $d$-chorded clutter. This means that $\lg \cc \rg$ is $d$-chorded.
\end{proof}

\begin{thm}\label{C-F chorded}
Let $\cc$ be a $d$-clutter. Suppose that $\lg \cc \rg$ is $d$-chorded (for example, if $\cc$ is chordal) and $F\in
\sms(\cc^+)$. Then $\lg \cc-F \rg$ is $d$-chorded.
\end{thm}
\begin{proof}
Suppose that $\Omega$ is a $d$-cycle of $\cc-F$ and $\cd= \F(\Omega)$. Since $\lg \cc \rg$ is $d$-chorded and
according to \ref{d-chorded}, there are complete subclutters $\cd_1, \ldots, \cd_t$ of $\cc$ on
$(d+2)$-subsets of $\V(\cd)$ such that $\cd= \oplus_{i=1}^t \cd_i$. We choose $\cd_i$'s in a way that $m=
|\{i|F\in \cd_i\}|$ is minimum possible. Assume that $m\geq 1$. Because $F\not \in \cd$, $m$ is at least 2.
Thus there are two $\cd_i$, say $\cd_1,\cd_2$, containing $F$. Therefore there exist $v_1 \neq v_2\in\V(\cd)$
such that $\V(\cd_i)= Fv_i$ for $i=1,2$. Let $W=Fv_1v_2$ and $\cc'= \cc_{W}$. Then since $F\in \sms(\cc^+)$
and $W\se \N_{\cc^+}(F)$, $\cc'$ is complete. Now if $\cd'= \cd_1 \oplus \cd_2$, then $\lg \cd' \rg$ is a
$d$-cycle in $\lg \cc'-F \rg$. By \ref{almost complete}, $\cc'-F$ is chordal and hence has a linear
resolution over every field. Thus by \cite[Theorem 6.1]{CF1} $\lg \cc'-F \rg$ is $d$-chorded. Consequently
according to \ref{d-chorded}, $\cd'=\cd'_1 \dis \cdots \dis \cd'_{t'}$, where $\cd'_i$'s are complete
subclutters of $\cc'-F$ (and hence $\cc$) on $(d+2)$-subsets of $\V(\cd')\se \V(\cd)$. By replacing
$\oplus_{i=1}^{t'} \cd'_i$ instead of $\cd_1 \oplus \cd_2$ in the decomposition of $\cd$, we get a
decomposition with less terms containing $F$. This contradicts the choice of the decomposition of $\cd$ and
hence $m=0$.  Therefore, the decomposition $\cd= \oplus_{i=1}^t \cd_i$ is in $\cc-F$ and the result follows
by \ref{d-chorded}.
\end{proof}

It should be mentioned that the previous result is not correct for arbitrary $F\in \MS(\cc^+)$. This can be
seen by noting that $I(\b{\cc-135})$ is not 2-chorded, where $\cc$ is as in Example \oldref{chorded octa}.

If $\cc$ is a graph, we call a $F\in \sms(\cc^+)$ a \emph{simplicial edge} of $\cc$.
\begin{cor}\label{chordal edge order}
Suppose that $\cc$ is a chordal graph and $F$ is a simplicial edge of $\cc$. Then $\cc-F$ is chordal. Hence there is an
ordering $F_1, \ldots, F_t, F_{t+1}, \ldots, F_m$ of edges of $\cc$ such that for $1\leq i\leq t$, $F_i$ is a
simplicial edge of the chordal graph $\cc_i= \cc-F_1-\cdots-F_{i-1}$ and for $t<i\leq m$, $\cc_i$ is a tree and $F_i$
is a leaf edge of $\cc_i$.
\end{cor}
\begin{proof}
The first statement is just \ref{C-F chorded} in the case that $\dim \cc=1$. Now if $\cc$ is chordal and $\cc^+\neq
\tohi$, then $\sms(\cc^+)\neq \tohi$ by \ref{chordal +}. So starting with $\cc$, we can delete simplicial edges until
we reach a chordal graph $\cc'$ with $(\cc')^+=\tohi$. But a chordal graph without any cliques on more than two
vertices is a tree and hence we can delete leaf edges from $\cc'$ until there is no more edges and the statement is
established.
\end{proof}

Let $\cc$ be the graph in \oldref{Fig}. Then $\cc^+$ and $\cc-F$ are chordal and also $I(\b{\cc^+})$ and
$I(\b{\cc-F})$ have linear quotients for each $F\in \sms(\cc^+)$. Despite this $\cc$ is not chordal, since
$\cc-e$ is a cycle of length 5. Thus the converses of \ref{C-F chorded} and \ref{chordal +}  do not hold.
Also edges of $\cc$ can be ordered as in \ref{chordal edge order}, so the converse of \ref{chordal edge
order} is not true, either.

\begin{figure}[!ht]
  \centering
  \includegraphics[scale=1.2]{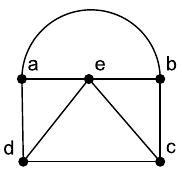}
  \caption{A non-chordal graph}\label{Fig}
\end{figure}

Clearly Statement \ref{D'} holds when $\dim \cc=1$. Using the concept of $d$-cycles, we show that \ref{D'}
holds when $\dim \cc^\vee \leq 1$ or equivalently, if $n-d\leq 3$. As in \cite{our chordal}, by a
\emph{CF-tree }we mean a $d$-clutter $\cc$ with the property that $\lg \cc \rg$ has no $d$-cycles.
\begin{prop}\label{n<=d+3}
Assume that $\cc$ is a $d$-clutter on $n$ vertices with $n\leq d+3$. If $\cc^+=\tohi$ and $I(\b\cc)$ has a linear
resolution over every field, then $\cc$ is chordal. In particular, statement \ref{D'} of \ref{main} holds for $\cc$.
\end{prop}
\begin{proof}
Suppose that $\cc^+=\tohi$  and $I(\b\cc)$ has a linear resolution over every field. According to
\cite[Theorem 6.1]{CF1}, $\lg \cc \rg$ is $d$-chorded. But since $\cc$ has no cliques of size $d+2$, it
follows from \ref{d-chorded} that $\lg \cc \rg$ has no $d$-cycles, that is, $\cc$ is a CF-tree. Now the
result follows from \cite[Corollary 3.7]{our chordal}.
\end{proof}
The proof of Corollary 3.7 of \cite{our chordal} uses Alexander dual. We end this paper mentioning that more generally,
one can get a statement equivalent to \ref{D'} by passing to the Alexander dual of $\lg \cc \rg$. Indeed, by arguments
quite similar to \cite[Theorem 3.6]{our chordal} one can see that \ref{D'} holds for all $d$-clutters on $n$ vertices,
\ifof statement (ii) of \cite[Theorem 3.6]{our chordal} holds, when we replace ``Cohen-Macaulay over $\z_2$'' in that
statement with ``shellable''.

\paragraph{Acknowledgements}
The author would like to thank Rashid Zaare-Nahandi and Ali Akbar Yazdan Pour for their helpful discussions
and comments. Also the author is thankful to the reviewer whose comments led to great improvements in the
presentation of the paper.

%%%%%%%%%%%%%%%%%%%%%%%%%%%%%%%%%%%%%%%%%%%%%%%%%%%%%%%%%%%%%%%%%%%%%%%%%%%%%%%%%%%%%%%%%%%%%%%%%%%%%%%%%%%%%%%%%%

\end{document}